\documentclass[12pt]{amsart}
\usepackage[latin1]{inputenc}
\usepackage{amssymb, amsmath, color, textcomp, url,latexsym}
\usepackage[hypertex]{hyperref}
\usepackage{enumerate}

\makeatletter
\@namedef{subjclassname@2010}{%
  \textup{2010} Mathematics Subject Classification}
\makeatother

\renewcommand{\marginpar}[1]{}

\numberwithin{equation}{section}

\newtheorem{theorem}{Theorem}[section]

\newtheorem{lemma}[theorem]{Lemma}
\newtheorem{fact}[theorem]{Fact}

\theoremstyle{definition}
\newtheorem{remark}[theorem]{Remark}
\newtheorem{definition}[theorem]{Definition}

\newtheorem*{question}{Question}

\newcommand{\cb}{\operatorname{Cb}}
\newcommand{\tp}{\operatorname{tp}}

\def\dcl{\mathrm{dcl}}

\def\cl{\mathrm{cl}}
\def\aclq{\mathrm{acl}^\mathrm{eq}}

\def\Par#1{\ulcorner#1 \urcorner}

\def\Ind#1#2{#1\setbox0=\hbox{$#1x$}\kern\wd0\hbox to 0pt{\hss$#1\mid$\hss}
\lower.9\ht0\hbox to 0pt{\hss$#1\smile$\hss}\kern\wd0}
\def\Notind#1#2{#1\setbox0=\hbox{$#1x$}\kern\wd0\hbox to 0pt{\mathchardef\nn="0236\hss$#1\nn$\kern1.4\wd0\hss}\hbox to 0pt{\hss$#1\mid$\hss}\lower.9\ht0
\hbox to 0pt{\hss$#1\smile$\hss}\kern\wd0}
\def\ind{\mathop{\mathpalette\Ind{}}}

\def\CML{$2$-tightness}
\def\cml{$2$-tight}
\def\CBP{$1$-tightness}
\def\cbp{$1$-tight}

\begin{document}

\baselineskip=17pt

\title{On Variants of CM-triviality}

\author[T. Blossier]{Thomas Blossier}
\address{Universit\'e de Lyon; CNRS; Universit\'e Lyon 1; Institut
Camille Jordan UMR5208, 43 boulevard du 11 novembre 1918, F--69622
Villeurbanne Cedex, France}
\email{blossier@math.univ-lyon1.fr}

\author[A. Martin-Pizarro]{Amador Martin-Pizarro}
\address{Universit\'e de Lyon; CNRS; Universit\'e Lyon 1; Institut
Camille Jordan UMR5208, 43 boulevard du 11 novembre 1918, F--69622
Villeurbanne Cedex, France}
\email{pizarro@math.univ-lyon1.fr} 

\author[F. Wagner]{Frank O. Wagner}
\address{Universit\'e de Lyon; CNRS; Universit\'e Lyon 1; Institut
Camille Jordan UMR5208, 43 boulevard du 11 novembre 1918, F--69622
Villeurbanne Cedex, France}
\email{wagner@math.univ-lyon1.fr}

\date{30 August 2012}

\begin{abstract}
We introduce a generalization of CM-triviality relative to a fixed invariant collection of partial types, in analogy to the Canonical Base Property defined by Pillay, Ziegler and Chatzidakis which generalizes one-basedness. We show that, under this condition, a stable field is internal to the family, and a group of finite Lascar rank has a normal nilpotent subgroup such that the quotient is almost internal to the family.
\end{abstract}



\maketitle

\section{Introduction}

Recall that a stable theory is called one-based if the canonical base of any real tuple $\bar a$ over an algebraically 
closed set $A$ is algebraic over $\bar a$. Hrushovski and Pillay have shown \cite{HrPi87} that in a stable one-based group, definable sets are finite unions of cosets of $\aclq(\emptyset)$-definable subgroups. Furthermore, the group is abelian-by-finite. In the finite Lascar rank context, the notion of one-basedness agrees with local modularity and also with $k$-linearity (i.e.\ the canonical parameter of any uniformly definable family of curves has Lascar rank at most $k$) for any $k>0$ (and in particular these notions agree for different $k$). 

Hrushovski, in his \emph{ab initio} strongly minimal set \cite{Hr93}, introduced a weaker notion than one-basedness, CM-triviality, of which Baudisch's new uncountably categorical group \cite{Ba96} is a proper example. Pillay \cite{Pi95} showed that neither infinite fields nor bad groups could be 
interpretable in a CM-trivial stable theory; he concluded that a CM-trivial group of finite Morley rank must be nilpotent-by-finite. In \cite{Pi00} Pillay defined a whole hierarchy of geometrical properties ($n$-ampleness for $n>0$), where one-based means non-$1$-ample, CM-trivial agrees with non-$2$-ample, and infinite fields are $n$-ample for all $n>0$. 

In \cite{PiZi03}, Pillay and Ziegler reproved the function field case of the Mordell-Lang conjecture in characteristic zero inspired by Hrushovski's original proof but avoiding the use of Zariski Geometries. Instead,
motivated by work of Campana \cite{Ca80} and Fujiki \cite{Fu}, 
they showed in particular that the collection of types of finite Morley rank in the theory of differentially closed fields in characteristic zero has the \emph{Canonical Base Property} (in short, CBP) with respect to the field of constants. That is, given a definable set $X$ of bounded differential degree and Morley degree $1$ in a saturated differentially closed field, the field of definition of the constructible set determined by $X$ is almost internal to the constant field over a generic realisation of $X$. Furthermore, they showed the 
CBP for certain types in the theory of separably closed fields as well as for types of finite rank in the theory of generic difference fields of characteristic zero. The latter result was subsequently generalised by Chatizidakis \cite{zC11} to all positive characteristics.

Clearly, the CBP is a generalisation of one-basedness, since for the latter the type of the canonical base of the generic type of $X$ is already algebraic over a realisation. For the CBP, we replace algebraicity by almost internality with respect to a fixed invariant collection $\Sigma$ of partial types. Chatizidakis \cite{zC11} has shown that for a theory of finite rank, the CBP already 
implies a strengthening of this notion, called UCBP, introduced by Moosa and Pillay \cite{MoPi08} in their study of compact complex spaces. 

Kowalski and Pillay \cite{KoPi06}, generalizing the definability results obtained for one-based theories, showed that a connected group definable in a stable theory with the CBP is central-by-(almost $\Sigma$-internal).

The goal of this note is to introduce an anlogous generalization of CM-triviality, again replacing algebraicity by almost $\Sigma$-internality. In view of the ampleness hierarchy, we call it \CML\ (and \CBP\ is just the CBP).
Similarly to the generalisation from one-based to the CBP, we show that a \cml\ stable field is $\Sigma$-internal, and a \cml\ group of finite Lascar rank is nilpotent-by-(almost $\Sigma$-internal). Recent work in \cite{BMPZ12} shows that the inclusion of the CM-trivial theories in the \cml\ ones is proper. 

\section{One-basedness and \CBP}\label{S:Defs}

The notions we will present could be generalised to any simple theory, replacing imaginary algebraic closure by bounded closure. However, the 
role of stability is essential for our main result, Theorem \ref{T:Main}. We shall thus restrict ourselves to the stable case. In the following we will be working inside a sufficiently saturated model of a stable theory $T$. Note that we allow our elements and types to be imaginary, unless stated otherwise.

\begin{definition}\label{D:1based}
The theory $T$ is \emph{one-based} if for every real set $A$ and every real tuple $c$ the canonical base $\cb(c/A)$ is algebraic over $c$. Equivalently, for every pair of real sets $A\subseteq B$ and every real tuple $c$, the canonical base $\cb(c/A)$ is algebraic over $\cb(c/B)$. 
\end{definition}

\begin{remark}\label{R:realim}
It was remarked in \cite{HrPi87} that if $T$ is one-based, the same conclusion holds for (possibly imaginary) sets $A$, $B$ and tuples $c$. As noted in \cite{Pi95}, we may assume that $A$ (and $B$) are small submodels.\end{remark}

Recall some of the results proved by Hrushovski and Pillay \cite{HrPi87}:
\begin{fact}\label{F:1based}
Let $G$ be an interpretable group in a one-based stable theory. Then:
\begin{enumerate}
\item Every connected definable subgroup $H$ of $G$ is definable over $\aclq(\emptyset)$. 
\item The connected component of $G$ is abelian. 
\end{enumerate}
\end{fact}

\begin{definition}\label{D:CM}
The theory $T$ is \emph{CM-trivial} if for every pair of real sets $A\subseteq B$ and every real tuple $c$, if $\aclq(Ac) \cap \aclq(B) = \aclq(A)$ then $\cb(c/A)$ is algebraic over $\cb(c/B)$. 
\end{definition}

Obviously, any one-based theory is CM-trivial.

\begin{remark}
As shown in \cite{Pi95}, Remark \ref{R:realim} holds for CM-trivial theories as well.
\end{remark}

Recall that a \emph{bad group} is a connected group of finite Morley rank which is non-solvable and whose proper connected definable subgroups are all nilpotent. It is unknown whether bad groups exist. A bad group of minimal rank can be taken simple, i.e.\ with no proper normal subgroups, by dividing out by its finite centre.

Pillay shows in \cite{Pi95} the following:
\begin{fact}\label{F:CM}
Let $T$ be a CM-trivial stable theory. Then:
\begin{enumerate}
\item No infinite field can be interpreted. 
\item An interpretable group of finite Morley rank is nilpotent-by-finite. In particular, no bad group can be interpreted. 
\end{enumerate}
\end{fact}

For the remainder of this article, we fix an $\emptyset$-invariant family $\Sigma$ of partial types.
\begin{definition}
A type $p$ over $A$ is \emph{internal to $\Sigma$}, or \emph{$\Sigma$-internal}, if for every realisation $a$ of $p$ there is some superset $B\supset A$ with $a\ind_A B$, and realisations $b_1,\ldots,b_r$ of types in $\Sigma$ based on $B$ such that $a$ is definable over $B,b_1,\ldots,b_r$. If we replace definable by algebraic, then we say that $p$ is \emph{almost internal to $\Sigma$} or \emph{almost $\Sigma$-internal}.\end{definition}

Poizat has shown that in a stable theory the parameter set $B$ can be fixed
independently of the choice of $a$.

\begin{definition}\label{D:CBP}
We say that the theory $T$ is \emph{\cbp} with respect to $\Sigma$ if for every tuple $c$ and every set $A$, the type $\tp(\cb(c/A)/ c)$ is almost $\Sigma$-internal.
\end{definition}

More generally, one can define \cbp\ for a collection of partial types :

\begin{definition}
A collection  $\mathcal{F}$ of partial types in the theory $T$ is \emph{\cbp} with respect to $\Sigma$ if for every set $A_0$ of parameters, every realisation $c$ of a tuple of types in $\mathcal{F}$ with parameters in $A_0$ and every $A\supseteq A_0$, the type $\tp(\cb(c/A)/A_0 c)$ is almost $\Sigma$-internal.\end{definition}

By definition, $T$ is \cbp\ if and only if the collection of all types over $\emptyset$ is. Since we can assume that $c$ contains the parameters $A_0$, this is also equivalent to the collection of all types being \cbp.

The {\em Canonical Base Property} CBP defined by Moosa, Pillay and Chatzidakis states that the collection of all types of finite Lascar rank is \cbp\ with respect to the family of types of Lascar rank one. Our terminology provides for a possible generalisation of non-ampleness for higher values of $n$.

\begin{remark}\label{R:CBPforAB}
Clearly, a one-based theory is \cbp\ with respect to any family $\Sigma$.
\end{remark}

\begin{lemma}\label{L:cbpCm} A theory $T$ is \cbp\ with respect to $\Sigma$ if and only if for every pair of real sets  $A\subset B$ and every real tuple $c$, the type $\tp(\cb(c/A)/\cb(c/B))$ is almost $\Sigma$-internal.\end{lemma} 
So Definition \ref{D:CBP} generalises either of the two equivalent notions of one-basedness.
\begin{proof} This is similar to \cite[Remark 2.1]{Pi95}. Note that the requirement on $A$, $B$ and $\bar c$ to be real is no restriction, as we can replace imaginaries by sufficiently independent real representatives.

One implication is trivial by setting $B=\aclq(Ac)$, since then $\cb(c/B)$ is interalgebraic with $c$. The other 
implication follows from the fact that 
$$c \ind_{\cb(c/B)} B\quad\text{yields}\quad c \ind_{\cb(c/B)} \cb(c/A)$$
and since a non-forking restriction of an almost internal type is again almost internal. 
\end{proof}

Adapting the proof of Fact \ref{F:1based}, Kowalski and Pillay proved the following 
\cite{KoPi06}.

\begin{fact}\label{F:CBP}
Let $G$ be a type-definable group in a stable theory which is \cbp\ with respect to $\Sigma$.
\begin{enumerate}
\item Given a connected type-definable subgroup $H\leq G$, the type of its canonical parameter $\tp(\Par{H})$ is almost $\Sigma$-internal.
\item If $G$ is connected, then $G/Z(G)$ is almost $\Sigma$-internal.
\end{enumerate}
\end{fact}

Hrushovski, Palac\'\i n and Pillay \cite{HPP12} exibit a theory of finite Morley rank without the CBP (i.e. which is not \cbp\ with respect to the family of all types of Lascar rank one), elaborating an example of Hrushovski.
However, if we replace almost internality by analysability, the resulting condition (strong $\Sigma$-basedness, or non weak $1$-$\Sigma$-ampleness in the terminology of \cite{PalWa11}) holds in any simple theory with enough regular types if we take $\Sigma$ to be the family of all non one-based regular types \cite[Theorem 5.1]{PalWa11}, generalising a result of Chatzidakis for types of finite SU-rank \cite{zC11}.

\begin{question}
Is every CM-trivial theory of finite rank \cbp\ with respect to the family of all types of rank $1$?
\end{question}

It is easy to see that if $T$ is \cbp\ with respect to $\Sigma$ and $\Sigma'$ results from $\Sigma$ by removing all one-based types, then $T$ is \cbp\ with respect to $\Sigma'$.

\section{Definable groups in \cml\ theories}

In this section, we will introduce a weaker notion than CM-triviality, called \CML, analogous to the generalisation from one-based to
\cbp. We will show that stable \cml\ fields and \cml\ simple groups of finite Lascar rank are $\Sigma$-internal. More generally, we 
show that a \cml\ group of finite Lascar rank is nilpotent-by-(almost $\Sigma$-internal). 

\begin{definition}\label{CM-linear}
The theory $T$ is called \emph{\cml} with respect to $\Sigma$ if for every pair of sets $A\subset B$ and every tuple $c$, if $\aclq(Ac) \cap\aclq(B) = \aclq(A)$, then $\tp(\cb(c/A)/\cb(c/B))$ is almost $\Sigma$-internal. 
\end{definition}

Again, one can also define \cml\ for a collection of partial types :

\begin{definition}
A collection  $\mathcal{F}$ of partial types in the theory $T$ is \emph{\cml} with respect to $\Sigma$ if for every set $A_0$ of parameters, every realisation $c$ of a tuple of types in $\mathcal{F}$ with parameters in $A_0$ and every $B\supseteq A\supseteq A_0$, if $\aclq(Ac) \cap\aclq(B) = \aclq(A)$, then $\tp(\cb(c/A)/A_0\cb(c/B))$ is almost $\Sigma$-internal.
\end{definition}

As before, $T$ is \cml\ with respect to $\Sigma$ if and only if the collection of all types is.

\begin{remark}
Every CM-trivial theory is \cml\ with respect to any family $\Sigma$. By Lemma 
\ref{L:cbpCm}, a \cbp\ theory with respect to $\Sigma$ is \cml\ with respect to $\Sigma$.

As in Remark \ref{R:realim}, we may assume that $A$ and $B$ in the definition are small submodels. If $T$ is \cml, the property also holds for imaginary tuples $c$. 
\end{remark}

\begin{definition}
A type $p$ is \emph{foreign to  $\Sigma$} if every non-forking extension of $p$ over $B$ is orthogonal to every extension over $B$ of any type in $\Sigma$. 
\end{definition}

We will repeatedly use the following connection between being foreign and internal for stable groups \cite[Theorem 3.1.1 and Corollary 3.1.2]{Wa97}

\begin{fact}\label{F:a.intoint}
\begin{enumerate}
\item If the generic type of a stable type-definable group $G$ is not foreign to $\Sigma$, 
then there is a relatively definable normal subgroup $N$ of infinite index such that $G/N$ is $\Sigma$-internal. 
\item A definably simple stable group or a stable division ring $G$ whose generic type is not foreign to $\Sigma$ must be $\Sigma$-internal.
In particular, if its generic type is almost $\Sigma$-internal then $G$ is $\Sigma$-internal.
\end{enumerate}
\end{fact}

Inspired by Fact \ref{F:CM} we show the following.

\begin{theorem}\label{T:Main} Let $T$ be stable and \cml\ with respect to $\Sigma$. 
\begin{enumerate}
\item An interpretable field $K$ is $\Sigma$-internal. 
\item An interpretable group $G$ of finite Lascar rank is 
nilpotent-by-(almost $\Sigma$-internal). In particular, an interpretable non-abelian simple group is $\Sigma$-internal.
\end{enumerate}
\end{theorem}

\begin{proof} $(1)$ We adapt the proof of \cite[Proposition 3.2]{Pi95}. Let $K$ be our interpretable field. Consider a generic plane $P$ given by $ax+by+c=z$, where $a,b,c$ are generic independent. Take a line $\ell$ contained in $P$ given by $y=\lambda x + \mu$ such that $\lambda,\mu$ are generic independent over $a,b,c$. Take a point $p=(x,y,z)$ in $\ell\subset P$ generic over $a,b,c,\lambda,\mu$. 

Pillay has shown that $\aclq(Ap) \cap B = A$, where $A=\aclq(\Par{P})$ is the imaginary algebraic closure of the canonical parameter of the plane and $B=\aclq(\Par{P}\Par{\ell})$ that of the parameters of the plane and the line. Furthermore, he has also shown that $\cb(p/A)= \Par{P}$ and $\cb(p/B)= \Par{\ell}$. By \CML, $\tp(\Par{P}/\Par{\ell})$ 
is almost $\Sigma$-internal. Note that the line $\ell$ is given by the equations 
$$\begin{aligned}
   y &=\lambda x + \mu \\
   z &= (a+b\lambda) x + (c+b\mu).
  \end{aligned}$$
Since $a,\lambda,\mu,b,c$ are generic independent, so are $a,\lambda,\mu,b\lambda+a,c+b\mu$.
Now 
$$a\in\dcl(\Par{P})\quad\text{and}\quad\Par{\ell}\in\dcl(a+b\lambda,c+b\mu,\lambda,\mu)\,;$$
as $\tp(\Par{P}/\Par{\ell})$ is almost $\Sigma$-internal, so is $\tp(a/a+b\lambda,c+b\mu,\lambda,\mu)$. Since 
$$a\ind a+b\lambda,c+b\mu,\lambda,\mu,$$
we obtain that $\tp(a)$ is almost $\Sigma$-internal. Thus $K$ must be $\Sigma$-internal by Fact \ref{F:a.intoint}
\vskip5mm
$(2)$ Let $G$ be our group, which we may assume infinite. We treat first the case where $G$ is simple non-abelian.
We may assume that every type-definable connected subgroup of $G$ is solvable: Otherwise we consider an infinite type-definable connected non-solvable subgroup $G_0$ of minimal Lascar rank, and a type-definable connected proper normal subgroup $H_0$ of $G_0$ of maximal rank. Then $H_0$ is solvable by minimality of rank, and contained in a relatively definable solvable proper normal subgroup $H_1$ of $G_0$ by stability. Moreover, every proper type-definable normal subgroup of $G_0/H_1$ is finite and hence central by connectedness, so $G_0/Z(G_0/H_1)$ is an infinite, type-definable, connected simple group of smaller Lascar rank than $G$. Inductively we may assume that this quotient is $\Sigma$-internal. But then the generic of $G$ is not foreign to $\Sigma$ by unidimensionality of simple groups of finite Lascar rank. Fact \ref{F:a.intoint} yields that $G$ is $\Sigma$-internal.
 
By superstability $G$ contains an infinite type-definable solvable connected subgroup. If there is one, $S$, which is not nilpotent, then there is a section of $G$ which interprets a field. By $(1)$, this field is $\Sigma$-internal, and so is $G$ by unidimensionality and Fact \ref{F:a.intoint}.

Otherwise, all proper connected type-definable subgroups of $G$ are nilpotent. Let $S$ be one of maximal Lascar rank and observe that $S$ intersects trivially any other maximal type-definable connected nilpotent subgroup $S_0$: Let $I=S\cap S_0$ be the intersection and assume $S_0$ is such that $I$ is non-trivial of maximal rank possible. Since both $S$ and $S_0$ are connected, the indices of $I$ in $S$ and in $S_0$ are both infinite. Hence so is the index of $I$ in $N_S(I)$ because $S$ is nilpotent \cite[Proposition 1.12]{Po85}. Simplicity of $G$ yields that $N_G(I)\not=G$. Therefore, $N_G(I)^0$ is nilpotent and contained in some maximal type-definable connected nilpotent subgroup $S_1$. But then $I< N_S(I)^0\le S_1\cap S$ yields an intersection of larger rank, whence $S_1=S$ by maximality. Now $I< N_{S_0}(I)^0\le S_1\cap S_0=S\cap S_0$ again provides an intersection of larger rank, 
a contradiction. In particular, $S$ intersects all its distinct conjugates trivially. Moreover, as $N_G(S)^0$ must also be 
nilpotent, maximality of $S$ implies that $S$ has bounded index in its normalizer. So
$$U(\bigcup_{g\in G}S^g)=U(G/N_G(S))+U(S)=U(G/S)+U(S)=U(G)$$
and hence every generic element of $G$ is contained in some conjugate of $S$.

We can now use the proof of \cite[Lemma 2.5]{Wa98}, which generalizes \cite[Lemma 3.4 and Proposition 3.5]{Pi95}, 
taking $\mathfrak B$ to be the family of $G$-conjugates of $S$, and $Q$ the family of all types of rank $1$ 
(so $Q$-closure $\cl_Q$ equals imaginary algebraic closure). It is shown in \cite{Wa98} that for $a,b,c$ in $G$ generic 
independent and a generic conjugate $S_0$ of $S$ over $a,b,c$, we have
$$\aclq(b,c,A)\cap B = A,$$ 
where $A=\aclq(\Par{(b,c)\cdot G_a})$ is the imaginary algebraic closure of the 
canonical parameter of the coset determined by $(b,c)$ of the subgroup 
$$G_a=\{(g,g^a)\mid g\in G\},$$
and $B=\aclq(\Par{P}\Par{l})$, where $l$ is the coset determined by $(b,c)$ of the subgroup 
$$S_a=\{(s,s^a)\mid s\in S_0\}.$$
Moreover, $\cb(b,c/A)= \Par{(b,c)\cdot G_a}$ and $\cb(b,c/B)= \Par{(b,c)\cdot S_a }$. By \CML, the type 
$\tp(\cb(b,c/A)/\cb(b,c/B))$ is almost $\Sigma$-internal.

Since $G$ is simple non-abelian, its center is trivial, so $a$ is definable over $\Par{G_a}$, which in turn is definable 
over $\Par{(b,c)\cdot G_a}=\cb(b,c/A)$. Furthermore, $\Par{(b,c)\cdot S_a}$ is definable over $\aclq(b,c,\Par{S_0}, 
\Par{a\cdot Z(S_0)})$. Hence $\tp(a/b,c,\Par{S_0},\Par{a\cdot Z(S_0)})$ is almost $\Sigma$-internal. As $b,c$ are generic 
independent from $a$ and the parameters of $S_0$, we have 
$$ b,c \ind_{\Par{S_0},\Par{a\cdot Z(S_0)}} a,$$
so the non-forking restriction $\tp(a/\Par{S_0},\Par{a\cdot Z(S_0)})$ is almost $\Sigma$-internal as well. 

Now, since $a$ is generic over $\Par{S_0}$, it is generic in $aZ(S_0)$ over 
$\Par{a\cdot Z(S_0)},\Par{S_0}$  by stability. So the generic type of $aZ(S_0)$ is almost $\Sigma$-internal, and so is the 
generic type of $Z(S_0)$. Since $S$ (and hence $S_0$) is nilpotent infinite, so is its center \cite[Proposition 1.10]{Po85}. 
Again unidimensionality of $G$ and Fact \ref{F:a.intoint} yield that $G$ is $\Sigma$-internal.

Let now $G$ be an arbitrary type-definable connected infinite group of finite Lascar rank in a \cml\ theory. By \cite{BP00}, there is a series of definable $G$-invariant subgroups
$$G=G_0\triangleright G_1\triangleright G_2\triangleright\cdots\triangleright G_n=\{1\}$$
such that all quotients $Q_i=G_i/G_{i+1}$ for $i<n$ are either finite, abelian, or simple. In fact, applying repeatedly 
Fact \ref{F:a.intoint} with respect to a single type of Lascar rank $1$ which is either a completion of a type in 
$\Sigma$ or foreign to $\Sigma$, we may refine the series and assume that every abelian quotient is unidimensional, and 
either $\Sigma$-internal or foreign to $\Sigma$. In the latter case, no definable section of the quotient is almost $\Sigma$-internal. Clearly all the finite quotients are $\Sigma$-internal; by the previous discussion, so are the simple quotients as well.

Let $C$ be the centraliser in $G$ of all 
$\Sigma$-internal quotients. If $Q_i$ is abelian, then trivially 
$(C\cap G_i)'\le C\cap G_{i+1}$; otherwise $C$ centralises $Q_i$ by definition, so again $(C\cap G_i)'\le C\cap G_{i+1}$. 
Thus $C$ is solvable. 

Suppose that $C^0$ is not nilpotent. Then there is a definable section $H$
of $C^0$ with a definable normal subgroup $K\triangleleft H$, such that $K$ is the additive group of an infinite 
interpretable field 
and and $H/K$ is isomorphic to an infinite multiplicative subgroup $T\le K^\times$, such that the multiplicative action of 
$T$ on $K^+$ comes from the action of $H/K$ on $K$ by conjugation. By \cite[Lemme 3.4]{Po85} 
(or \cite[Proposition 3.7.7]{Wa97} for the case of finite Lascar rank), the subgroup $T$ generates $K$ additively. 
It follows that $K$ has no non-trivial $H$-invariant subgroup: If $K_0\le K$ is $H$-invariant and $0\not=a\in K_0$, then 
$a^{-1}K_0$ is also $H$-invariant and contains $T$. As it is closed under addition, $a^{-1}K_0\ge\langle T\rangle^+=K$, 
whence $K_0=K$. Note that $K$ is $\Sigma$-internal by part (1).

Let $\tilde H$ and $\tilde K$ be the preimages of $H$ and $K$ in $C$, and let $i<n$ be minimal such that $\tilde K\le G_i$. Since 
$\tilde K\cap G_{i+1}$ is normal in $\tilde H$, it induces an $H$-invariant subgroup of $K$, which must be trivial. 
Therefore $K$ embeds definably into $Q_i$,  which must be almost $\Sigma$-internal as $K$ is. But then $Q_i$ is centralised by 
$C$, so $K$ is centralised by $H$. But this means that the multiplicative action of $T$ on $K^+$ is trivial, a contradiction.

Finally, stability yields that $C_G(Q_i)=C_G(\bar q_i)$ for some 
finite tuple $\bar q_i$ in $Q_i$. Hence any element $gC_G(Q_i)$ in $G/C_G(Q_i)$ is definable over the tuple $\bar q_i\bar q_i^g\in Q_i$. Therefore $G/C_G(Q_i)$ is $Q_i$-internal, and thus $\Sigma$-internal. Now $G/C$ embeds definably into the product 
$$\prod\{G/C_G(Q_i):Q_i\text{ $\Sigma$-internal}\}$$
via the map 
$$gC\mapsto (gC_G(Q_i):Q_i\text{ $\Sigma$-internal}).$$
Therefore $G/C$ is $\Sigma$-internal, and $G/C^0$ is almost $\Sigma$-internal.
\end{proof}

\begin{remark} If $\Sigma$ contains all non one-based types of rank $1$, 
\cite[Theorem 6.6]{PalWa11} implies in particular that  a group $G$ of finite Morley rank has a nilpotent subgroup $N$ such that $G/N$ is almost $\Sigma$-internal, without any hypothesis of \CML. Hence part $(2)$ above only provides additional information if $\Sigma$ does not contain all non one-based types of rank $1$. Note that if $\Sigma$ contains
 all $2$-ample types, then any simple theory is {\em weakly} \cml\ with respect to $\Sigma$ by \cite[Theorem 5.1]{PalWa11},
 where we replace almost internality by analysability in the definition.
\end{remark}

\begin{question}
What can one say about groups of infinite rank, or even merely stable ones? For instance, in a \cml\ stable group, can one find a nilpotent normal subgroup such that the quotient is almost $\Sigma$-internal?
\end{question}

\begin{remark}
The theory $T_{FPS}$ of the free pseudospace \cite{BP00} is not CM-trivial, and its collection of finite rank types is not \cbp\ with respect to the family of rank $1$ types (i.e.\ $T_{FPS}$ does not have the CBP). However $T_{FPS}$ is \cml\ with respect to the family of rank $1$ types \cite{BMPZ12}. Unfortunately, $T_{FPS}$ is trivial and therefore no infinite group is interpretable. Thus this example provides little insight into the significance of the above theorem. 

Any almost strongly minimal theory is \cbp\ and thus \cml\ with respect to the family of rank $1$ types; an example of a \cml\ non \cbp\ structure interpreting infinite groups is thus given by the disjoint union of the free pseudospace with an almost strongly minimal group. However, we don't know any example of a non-trivial \cml\ non \cbp\ theory which does not decompose into a trivial and a \cbp\ part.
\end{remark}

\subsection*{Acknowledgements} Research partially supported by the Agence Na\-tio\-nale pour la Recherche under contract ANR-09-BLAN-0047 Modig. The second author was supported by an Alexander von Humboldt-Stiftung For\-schungsstipendium f\"ur erfahrene Wissenschaftler 3.3-FRA/1137360 STP.

\end{document}